\documentclass[11pt]{amsart}
\usepackage[utf8]{inputenc}      % \PrerenderUnicode{ç}
\usepackage[T1]{fontenc}         % enables hyphenation of accented words
\usepackage{amsmath}
\usepackage{amssymb}
\usepackage{amsthm}  \swapnumbers
\usepackage{textcomp}
\usepackage[mathscr]{eucal}
\usepackage{graphicx}
\usepackage{epigraph}
\usepackage{paralist}
\usepackage[dvipsnames]{xcolor}  % load before tikz to get the option
\usepackage{tikz}
\usepackage{tikz-cd}
\usetikzlibrary{arrows}

\makeatletter
  \def\swappedhead@plain#1#2#3{%
    \thmnumber{(\textup{#2})}%                    % upright number
    \thmname{\@ifnotempty{#2}{~}\textup{#1}}%     % upright name
    \thmnote{ {\textup{(#3)}}}}%                  % upright note
  \let\swappedhead\swappedhead@plain
\makeatother

\theoremstyle{plain}
\newtheorem{theo}[equation]{Theorem}%[section]
\newtheorem{prop}[equation]{Proposition}

\newtheorem{algo}[equation]{Algorithm}
\theoremstyle{definition}
\newtheorem{defi}[equation]{Definition}
\newtheorem{exem}[equation]{Example}
\newtheorem{exems}[equation]{Examples}
\newtheorem{rema}[equation]{Remark}
\newtheorem{remas}[equation]{Remarks}
\numberwithin{equation}{section}

\usepackage{Z-fonts}
\usepackage{Z-upright}

\usepackage{bm}     % bold math, must come *after* all font declarations

\usepackage[backend=biber,bibstyle=Z-bib,citestyle=alphabetic]{biblatex}

\usepackage[colorlinks=true,%
            linkcolor=RoyalBlue,%
            citecolor=RedViolet,%
            urlcolor=RoyalBlue]{hyperref}   % hyperref, must come *last*

\let\polishl\l

  \newcommand\CC{{\mathbf{C}}}         % the complex numbers
  \newcommand\F{{\mathscr F}}          % système d'imprimitivité
  \newcommand\RR{{\mathbf{R}}}         % the reals
\renewcommand\o{^{\smash{\mathrm o}}}  % identity component

  \newcommand\Lie{\mathfrak}
  \newcommand\LA{{\Lie{a}}}            % Lie algebra of A
  \newcommand\LB{{\Lie{b}}}            % Lie algebra of B
  \newcommand\LD{{\Lie{d}}}            % Lie algebra of D
  \newcommand\LE{{\Lie{e}}}            % Lie algebra of E
  \newcommand\LF{{\Lie{f}}}            % Lie algebra of F
  \newcommand\LG{{\Lie{g}}}            % Lie algebra of G
  \newcommand\LH{{\Lie{h}}}            % Lie algebra of H
  \newcommand\LK{{\Lie{n}_c}}            % Lie algebra of K
  \newcommand\LN{{\Lie{n}}}            % Lie algebra of N
  \newcommand\LQ{{\Lie{q}}}            % Lie algebra of Q
  \newcommand\LS{{\Lie{s}}}            % Lie algebra of S
  \newcommand\LU{{\Lie{u}}}            % Lie algebra of U

  \DeclareMathOperator\ad{ad}          % adjoint representation
  \DeclareMathOperator\ann{ann}        % annihilator
  \DeclareMathOperator\drag{drag}      % symplectic gradient
\renewcommand\d{{\delta}}              % tangent vector
  \DeclareMathOperator\Diff{Diff}      % diffeomorphisms
  \newcommand\e[1]{{\mathrm e^{#1}}}   % exponential map
  \let\Imaginary\Im
  \let\Im\relax
  \DeclareMathOperator\Im{Im}          % imaginary part, image
  \DeclareMathOperator\Ind{Ind}        % induction
  \newcommand\ind{_{\textup{ind}}}
  \newcommand\inv{^{-1}}               % inverse
  \DeclareMathOperator\Isom{Isom}      % isomorphisms
  \DeclareMathOperator\Ker{Ker}        % kernel
  \newcommand\orth{^{\smash{\scriptscriptstyle\perp}}}
  \let\Real\Re
  \let\Re\relax
  \DeclareMathOperator\Re{Re}          % real part
  \DeclareMathOperator\Tr{Tr}          % real part  
  \newcommand\suba{_{\smash{|}\LA}}
  \newcommand\subh{_{\smash{|}\LH}}
  \newcommand\subk{_{\smash{|}\LK}}
  \newcommand\subn{_{\smash{|}\LN}}
  \newcommand\subs{_{\smash{|}\LS}}
  \newcommand\subu{_{\smash{|}\LU}}

  \newcommand{\<}{\langle}
\renewcommand{\>}{\rangle}
  \DeclareFontFamily{OMX}{MnSymbolE}{}  
  \DeclareFontShape{OMX}{MnSymbolE}{m}{n}{
      <-6>  [0.85] MnSymbolE5
     <6-7>  [0.85] MnSymbolE6
     <7-8>  [0.85] MnSymbolE7
     <8-9>  [0.85] MnSymbolE-Bold8
     <9-10> [0.85] MnSymbolE-Bold9
    <10-12> [0.85] MnSymbolE-Bold10
    <12->   [0.85] MnSymbolE12}{}

\renewcommand\iff{\Leftrightarrow}
  
  \newcommand\sub[1]{_{\smash{\raisebox{1pt}{$\scriptstyle #1$}}}}
  \newcommand\x{\times}

\providecommand\greektt{}
\providecommand\omicron{o}

\providecommand\Rho{{P}}

\showhyphens{Heisenberg hamiltonian}

\showhyphens{decomposition diffeomorphism equivariance imprimitivity isomorphism Kazhdan Lusztig module modules momentum representation represents represented Spaltenstein submersion submersions theoretic Williams}

\bibliography{biblio}

\title[Symplectic Mackey Theory]{Symplectic Mackey Theory}

\author{François Ziegler}
\email{fziegler@georgiasouthern.edu}
\address{Department of Mathematical Sciences, Georgia Southern University, Statesboro, GA 30460-8093, USA}

\date{\today}
\subjclass[2010]{53D20, 17B08, 22D30, 57S20}

\begin{document}

\begin{abstract}
   Many years ago Kazhdan, Kostant and Sternberg defined the notion of inducing a hamiltonian action from a Lie subgroup. In this paper, we develop the attendant imprimitivity theorem and Mackey analysis in the full generality needed to deal with arbitrary closed normal subgroups. 
\end{abstract}

\maketitle

% see http://article.gmane.org/gmane.comp.tex.tetex.general/688 :
% \let\languagename\relax   
\setcounter{tocdepth}{1}
\tableofcontents

\section*{Introduction}

The notion of an \emph{induced object}, characterized by a \emph{system of imprimitivity}, has its origins in Frobenius's early work on finite group representations (see \cite{Clifford:1937a}). It has since swarmed out across many disciplines: representations of locally compact groups \cite{Mackey:1949}, Fell bundles \cite{Takesaki:1967,Fell:1969a,Green:1978,Kaliszewski:2013}, Lie algebras \cite{Blattner:1969}, $C^*$\nobreakdash-algebras \cite{Rieffel:1974}, rings \cite{Rieffel:1975}, Hopf algebras \cite{Koppinen:1977}, ergodic actions of Lie groups \cite{Zimmer:1978}, algebraic groups \cite{Cline:1983}, hypergroups \cite{Colojoara:1988}, and quantum groups \cite{Vaes:2005}. Wherever it makes sense, the main application is a version of the \emph{normal subgroup analysis} of Clifford, Mackey, Blattner, and Fell \cite{Clifford:1937a,Mackey:1958,Blattner:1965,Fell:1969a}.

So it has been expected, ever since Kazhdan, Kostant and Sternberg \cite{Kazhdan:1978} defined inducing a hamiltonian action from a subgroup, that an attendant imprimitivity theorem and normal subgroup analysis should exist. In fact \cite[p.\,1]{Kostant:1968} already mentions, among eight projects, that of symplectically understanding ``Mackey's theory for the case of semi-direct products''. Now, one might wonder what the point is to mimic this theory in symplectic geometry: didn't Kirillov precisely \emph{integrate it out} of representation theory, with his direct parametrization of the unitary dual by coadjoint orbits? Our argument here is that, as Duflo \cite{Duflo:1981} points out, 
\begin{quote}\small
	in practice, the computation of $\LN^*/N$ is usually done by an inductive procedure parallel to Mackey's inductive procedure. This is even the heart of Kirillov's proof.
\end{quote}
The goal of this paper is to spell out that parallel theory, in the full generality needed to deal with arbitrary normal subgroups. In particular, to ensure recursive applicability of the resulting ``Mackey machine'', where subgoups typically arise as stabilizers, it is essential that we avoid any connectivity assumptions on our groups and orbits; this causes subtleties which, in the last (``Mackey obstruction'') step, were overcome only recently \cite{Iglesias-Zemmour:2015}.

Chapter I presents the imprimitivity theorem of \cite{Ziegler:1996} \eqref{imprimitivity_theorem}. This was already given an exposition in \cite[pp.\,332--336, 472]{Landsman:1998}, where Landsman noted that \emph{if $G$ and $H$ are connected and simply connected}, it can be deduced from the Morita equivalence theory of \cites[Thm 3.3]{Xu:1991}[p.\,322]{Landsman:1998}, much like Rieffel \cite[§7]{Rieffel:1974} deduced Mackey's theorem from an ``abstract'' version. While some of Xu's connectivity conditions can be removed \cites{Landsman:1998}[§8]{Landsman:2006}, the proof we shall give appears to be the only one valid unconditionally. In addition we prove a symplectic analog \eqref{intertwining} of Mackey's theorem on intertwining operators, and in \eqref{homogeneous_case} we characterize systems of imprimitivity on homogeneous hamiltonian $G$-spaces as all arising from \emph{Pukánszky coisotropic subalgebras}. The key role of such subalgebras was first brought out in \cite{Duflo:1982a} and, in relation with symplectic induction, \cite{Duval:1992}. 

Chapter II builds the symplectic Mackey machine, in a parallel with the three representation-theoretic steps well described in \cite[{}XII.1.28]{Fell:1988b}: \begin{itemize}[--]
	\item Given a closed normal subgroup $N\subset G$, Step 1 is just the observation that homogeneous hamiltonian $G$-spaces sit above $G$-orbits in $\LN^*/N$ \eqref{little_group}.
	\item Given such an orbit $\mathcal U=G(U)$, Step 2 uses the inducing construction and the imprimitivity theorem to classify $G$-spaces $X$ over $\mathcal U$ in terms of $G_U$-spaces $Y$ sitting above $U$ alone \eqref{induction_step}.
	\item Given $U=N(c)$, Step 3 uses a twisted product construction and the barycentric decomposition theorem of \cite{Iglesias-Zemmour:2015} to classify $G_U$\nobreakdash-spaces $Y$ over $U$ in terms of $G_c$-spaces $V$ sitting above $c_{|\LN_c}$ \eqref{obstruction_step}. 
\end{itemize}
Equivalently, the spaces $V$ in the last step are arbitrary homogeneous hamiltonian $(G_c/N\sub c\o, [-\theta])$-spaces, where $[\theta]$ is the \emph{symplectic Mackey obstruction} of $U$. We note that these results were anticipated by Lisiecki \cite{Lisiecki:1992} in the case where $G$ is nilpotent, connected and simply connected.

Chapter III illustrates the theory with three simple applications. In \eqref{Mackey_Wigner} we spell out the case of a connected \emph{abelian} normal subgroup, in which Step 3 above is essentially unnecessary. In the semidirect product case the result was known to Guillemin and Sternberg \cite[Thm 4.1]{Guillemin:1983} and in another form (before the invention of symplectic induction) to Rawnsley \cite[Prop.~1]{Rawnsley:1975}; in general it is clearly anticipated in Kirillov \cite[Lem.~4 \&~5]{Kirillov:1968}. In \eqref{Kirillov_Bernat} we use the Mackey machine inductively to show that coadjoint orbits of exponential groups are all ``monomial'', i.e.~induced from point-orbits. This result can be viewed as an independent proof of the existence of Pukánszky polarizations. Finally \eqref{parabolic_induction} applies the imprimitivity theorem to exhibit the coadjoint orbits of reductive groups as always ``parabolically induced''. When the parabolic is \emph{minimal} (Borel subgroup) this result goes back again to Guillemin and Sternberg \cite[Thm 3.1]{Guillemin:1983}.

\specialsection*{\textbf{Chapter I: The Imprimitivity Theorem}}

\section{Symplectic Preliminaries}

\subsection{Notation for group actions}

Whenever $G$ is a Lie group, we write $G\o$ for its identity component and  reserve the corresponding german letter for its Lie algebra, $\LG$. If $G$ acts on a manifold $X$, so that we have a morphism $g\mapsto g_{\,X}$ of $G$ into the diffeomorphisms of $X$ (with $g_{\,X}(x)$ a smooth function of the pair $(g,x)$), we define the corresponding infinitesimal action
\begin{equation}
	Z\mapsto Z_X,\qquad\qquad \LG\to\textrm{vector fields on $X$}
\end{equation}
by $Z_X(x) = \left.\frac{d}{dt}\exp(tZ)_X(x)\right|_{t=0}$. This is a Lie algebra morphism, if we define the bracket of vector fields with minus its usual sign. Whenever possible, we drop the subscripts to write $g(x)$ and $Z(x)$ instead of $g_{\,X}(x)$ and $Z_X(x)$. We use
\begin{equation}
	G(x),\qquad \LG(x),\qquad G_x,\qquad \LG_x,
\end{equation}
to denote respectively the $G$-orbit of $x$, its tangent space at $x$, the stabilizer of $x$ in $G$, and the stabilizer of $x$ in $\LG$. Finally it will be convenient to have a concise notation for the translation of tangent and cotangent vectors to $G$. Thus for fixed $g,q\in G$ we will let
\begin{equation}
   \label{lifted_actions}
	\gathered T_qG             \\\vspace{-3pt}       v       \endgathered
	\gathered  \,\to\,        \\\vspace{-3pt}   \,\mapsto\,  \endgathered
	\gathered T_{gq}G         \\\vspace{-3pt}      gv,       \endgathered
	\quad\qquad\text{resp.}\qquad\quad
	\gathered T^*_qG         \\\vspace{-3pt}       p         \endgathered
	\gathered  \,\to\,        \\\vspace{-3pt}  \,\mapsto\,   \endgathered
	\gathered T^*_{gq}G     \\\vspace{-3pt}      gp          \endgathered
\end{equation}
denote the derivative of $q\mapsto gq$, respectively its contragredient so that $\<gp,v\>=\<p,g^{-1}v\>$. Likewise we define $vg$ and $pg$ with $\<pg,v\>=\<p,vg^{-1}\>$. With this understood, the coadjoint action on $\LG^*=T^*_eG$ is $g(m)=gmg\inv$; infinitesimally it gives $Z(m)=\<m,[\,\cdot\,,Z]\>$. 

\subsection{Hamiltonian $G$-spaces}

Let $(X,\sigma)$ be a symplectic manifold. To each $f\in C^\infty(X)$ we attach the vector field $\drag f$ defined by $\sigma(\drag f,\cdot) = -df$ (`symplectic gradient'); $\drag$ is a Lie algebra morphism if we endow ${\text C}^\infty(X)$ with \emph{Poisson bracket}:
\begin{equation}
   \label{Poisson_bracket}
	\{f,f'\} = \sigma(\drag f',\drag f).
\end{equation}
The $\sigma$-preserving action of a Lie group $G$ on $X$ is called \emph{hamiltonian} if there is a \emph{moment map}, $\Phi:X\to\LG^*$, such that $Z_X=\drag\<\Phi(\cdot),Z\>$. If further $\Phi$ is $G$-equivariant, then $Z\mapsto \<\Phi(\cdot),Z\>$ is a Lie algebra morphism and the triple $(X,\sigma,\Phi)$ is called a \emph{hamiltonian $G$-space}. The notion of isomorphic hamiltonian $G$-spaces is clear:
\begin{equation}
   \Isom_G(X_1,X_2)
\end{equation}
will denote the set of all $G$-equivariant diffeomorphisms $X_1\to X_2$ which transform $\sigma_1$ into $\sigma_2$ and $\Phi_1$ into $\Phi_2$.

\begin{exems}
   \label{T*Q_KKS}
	(a) If $G$ acts on a manifold $Q$, we get an action on $X=T^*Q$ which preserves the canonical 1-form $\theta=$ ``$\<p,dq\>$'' and $\sigma=d\theta$. The relation $\iota(Z_X)d\theta + d\iota(Z_X)\theta=0$ then shows that a moment map $\Phi:X\to\LG^*$, which one can check is $G$-equivariant, is given by $\<\Phi(\cdot),Z\>=\iota(Z_X)\theta$, i.e.
	\begin{equation}
	   \label{moment_T*Q}
	   \<\Phi(p),Z\> = \<p,Z(q)\>,
	   \qquad
	   p\in T^*_qQ.
	\end{equation}
	(b) If $M$ is an orbit in $\LG^*$ for the coadjoint action of $G$, then the 2-form defined on it by $\sigma_{KKS}(Z(m),Z'(m)) =\<Z(m),Z'\>$ makes $(M,\sigma_{KKS},M\hookrightarrow\LG^*)$~into a homogeneous hamiltonian
	$G$-space. Conversely, every such space covers a coadjoint orbit:
\end{exems}

\begin{theo}[Kirillov-Kostant-Souriau]
   \label{KKS_theorem}
	Let $(X,\sigma,\Phi)$ be a hamiltonian $G$-space\textup, and suppose that $G$ acts transitively on $X$\textup. Then $\Phi$ is a symplectic covering of its image\textup, which is a coadjoint orbit of $G$.
\end{theo}

By \emph{symplectic} covering we mean of course that $\Phi$ pulls the orbit's 2-form back to $\sigma$. That $\Phi$ is a covering follows from the first of two informative properties, valid for any moment map:
\begin{equation}
   \label{Ker_Im_DPhi(x)}
   \Ker(D\Phi(x)) = \LG(x)^\sigma,
   \qquad\qquad
   \Im(D\Phi(x)) = \ann(\LG_x).
\end{equation}
Here the superscript means orthogonal subspace relative to $\sigma$, and if $(\cdot)$ is a subset of either $\LG$ or $\LG^*$, $\ann(\cdot)$ denotes its annihilator in the other.

\section{Symplectic Induction}
\label{Symplectic_induction}

\subsection{The Kazhdan-Kostant-Sternberg construction \cite{Kazhdan:1978}}
\label{KKS}
Given a closed subgroup $H$ of $G$ and a hamiltonian $H$-space $(Y,\tau,\Psi)$, this construction produces a hamiltonian $G$-space $(\smash{\Ind_H^GY},\sigma\ind,\Phi\ind)$ as follows. (We use the notation \eqref{lifted_actions}.)

First endow $N=T^*G\x Y$ with the symplectic form $\omega=d\theta+\tau$, where $\theta$ is the canonical 1-form on $T^*G$, and let $H$ act on $N$ by $h(p,y)=(ph\inv,h(y))$. This is hamiltonian, with moment map $\psi$:
\begin{equation}
   \label{psi}
	\psi(p,y) = \Psi(y) - q\inv p\subh
\end{equation}
for $p\in T^*_qG$. Here the second term denotes the restriction to $\LH$ of $q\inv p\in\LG^*$; it comes from \eqref{moment_T*Q} with $Z(q)=-qZ$. The induced manifold is now defined as the Marsden-Weinstein reduction of $N$ at zero \cite{Marsden:1974}, i.e.,
\begin{equation}
   \label{Ind_H^GY}
	\Ind_H^GY:=\psi\inv(0)/H.
\end{equation}
In more detail: the action of $H$ is free and proper (because it is free and proper on the factor $T^*G$, where it is the right action of $H$ regarded as a subgroup of $T^*G$ \cite[§III.1.6]{Bourbaki:1972}); so $\psi$ is a submersion (\ref{Ker_Im_DPhi(x)}b), $\psi\inv(0)$ is a submanifold, and \eqref{Ind_H^GY} is a manifold; moreover $\omega_{|\psi\inv(0)}$ degenerates exactly along the $H$-orbits (\ref{Ker_Im_DPhi(x)}a), so it comes from a uniquely defined symplectic form, $\sigma\ind$, on the quotient.

To make \eqref{Ind_H^GY} into a $G$-space, we let $G$ act on $N$ by $g(p,y)=(gp,y)$. This action commutes with the previous $H$-action and preserves $\psi\inv(0)$. Moreover it is hamiltonian and its moment map $\varphi:N\to\LG^*$, given by \eqref{moment_T*Q} with now $Z(q)=Zq$:
\begin{equation}
   \label{phi}
	\varphi(p,y) = pq\inv,      
	\qquad
	p\in T^*_qG,
\end{equation}
is constant on each $H$-orbit. Passing to the quotient, we obtain the required $G$-action on $\Ind_H^GY$ and moment map $\Phi\ind:\Ind_H^GY\to\LG^*$.

\subsection{Elementary properties}
The following includes a version of \emph{Frobenius reciprocity} (b) and the \emph{stages theorem} (e). In (c,d) we say that a hamiltonian $G$\nobreakdash-space \emph{is homogeneous} if $G$ acts transitively on it, and \emph{is a coadjoint orbit} if further the covering \eqref{KKS_theorem} is injective.

\begin{prop} 
   \label{elementary}
   \ 
   \begin{enumerate}[\upshape(a)]
      \item $\dim(\Ind_H^GY) = 2\dim(G/H) + \dim(Y)$.
      \item A coadjoint orbit $M$ of $G$ intersects $\Im(\Phi\ind)\iff\smash{M\subh}$ intersects $\Im(\Psi)$.
      \item If $\Ind_H^GY$ is homogeneous\textup, then $Y$ is homogeneous.
      \item If $\Ind_H^GY$ is a coadjoint orbit\textup, then $Y$ is a coadjoint orbit.
      \item If $K$ is an intermediate closed subgroup\textup, then $\Ind_K^G\Ind_H^KY=\Ind_H^GY$.
   \end{enumerate}
\end{prop}

\begin{proof}
   (a): \eqref{Ind_H^GY} has dimension $\dim(N) - \dim(\LH^*) - \dim(H)$ because $\psi$ is a submersion and $H$ acts freely.
   (b): This reexpresses $\Im(\Phi\ind)=\varphi(\psi\inv(0))$.
   (c): Assume $G$ is transitive on $\smash{\Ind_H^GY}$, and let $y_1, y_2\in Y$. Pick $m_i\in\LG^*$ such that $\Psi(y_i)=\smash{m_{i|\LH}}$. Then the $H$-orbits $x_i=H(m_i,y_i)$ are points in \eqref{Ind_H^GY}. So transitivity says that $x_1=g(x_2)$, i.e.
   \begin{equation}
      \label{transitivity}
	   (m_1,y_1)=(gm_2h\inv,h(y_2))
	   \quad\text{for some }
	   h\in H.
   \end{equation}
   In particular $y_1=h(y_2)$, as claimed.
   (d): Assume further that $\Phi\ind$ is injective, and that $\Psi(y_1)=\Psi(y_2)$. Then we can pick $m_1=m_2$ above. Since $\Phi\ind(x_i)=m_i$ it follows by injectivity that $x_1=x_2$, i.e. we have \eqref{transitivity} with $g=e$. But then $h=e$ and hence $y_1=y_2$, as claimed.
   (e): The left-hand side is by construction a space of $K\x H$-orbits within $T^*G\x T^*K\x Y$, and it is not hard to verify that an isomorphism from left to right is obtained by sending the $K\x H$-orbit of $(p,p',y)$, $p'\in T^*_{q'}K$, to the $H$-orbit of $(pq',y)$.
\end{proof}

\section{Symplectic Imprimitivity}

\subsection{Systems of imprimitivity}
Let $(X,\sigma,\Phi)$ be a hamiltonian $G$-space. The natural action of $G$ on  $C^\infty(X)$ will be denoted: $g(f)=f(g\inv(\cdot))$; it preserves the Poisson bracket \eqref{Poisson_bracket}.

\begin{defi}
   \label{imprimitivity}
   A \emph{\textbf{system of imprimitivity}} on $X$ is a $G$-invariant, Poisson commutative subalgebra $\LF$ of $C^\infty(X)$, such that $\drag f$ is complete for each $f\in\LF$.
\end{defi}

Given such a system, we write $\LF^*$ for its algebraic dual, and $\F$ for $\LF$ viewed as an additive group. Then $G$ acts on $\LF^*$ by contragredience, and $\F$ acts on $X$ by exponentiation of the fields $\drag f$. Although $\F$ is not usually a Lie group, we can still write $\F(x)$ for the $\F$-orbit of $x$, $\LF(x) = \{(\drag f)(x):f\in\LF\}$, and regard as \emph{moment} of this action the map
\begin{equation}
   \label{moment_pi}
   \pi:X\to\LF^*,
   \qquad\qquad
   \<\pi(x),f\> = f(x).
\end{equation}
The set $B = \pi(X)$ will be referred to as the \emph{base} of $\LF$. Since \eqref{moment_pi} is clearly $G$-equivariant, $B$ is in general a union of $G$-orbits. Whence:

\begin{defi}
   \label{transitive}
   The system of imprimitivity $\LF$ is \emph{\textbf{transitive}} if
   \begin{enumerate}[(i)]
      \item the $G$-action on its base $B=\pi(X)$ is transitive;
      \item $\pi:X\to B$ is $C^\infty$ for the $G$-homogeneous manifold structure of $B$. 
   \end{enumerate}
\end{defi}

\begin{remas}
   (a) Such is for instance automatically the case if $G$ is transitive on $X$ itself. We do not know if condition (i) might always imply (ii).
   
   (b) The manifold structure in (ii) is well-defined, for the stabilizer $G_b$ of any point $b=\pi(x)$ of $B$ is \emph{closed}. Indeed, $g\in G_b$ means that one has $\<g(b),f\>=\<b,f\>$, i.e. $f(g(x)) = f(x)$, for each $f\in\LF$; and this condition is closed by continuity of the maps $g\mapsto f(g(x))$.
   
   \label{same_base}
   (c) We will still call \emph{base of} $\LF$ any $G$-set with a fixed $G$-equivariant bijection onto $B$. This lets us speak of systems of imprimitivity having the \emph{same base}.
\end{remas}

\subsection{The system of imprimitivity attached to an induced manifold}
If $X=\smash{\Ind_H^GY}$ (§\ref{KKS}), then one obtains a $G$-equivariant projection
\begin{equation}
   \label{pi_ind}
   \pi\ind:\Ind_H^GY\to G/H
\end{equation}
by observing that the map $T^*G\x Y\to G/H$ which sends $T^*_qG\x Y$ to $qH$ is constant on $H$-orbits and hence passes to the quotient \eqref{Ind_H^GY}. There results a canonical sytem of imprimitivity on $X$:

\begin{prop}
   \label{f_ind}
   If $X = \Ind_H^GY$\textup, then $\LF\ind:=\pi\ind^*(C^\infty(G/H))$ is a transitive system of imprimitivity on $X$ with base $G/H$.
\end{prop}

\begin{proof}
   Let $f\in C^\infty(G/H)$, and write also $f$ for its pull-back to $X$ via \eqref{pi_ind}, to $G$, or to $T^*G\x Y$. Then its symplectic gradient $\drag f$ on the latter space is the vector field $\eta$ with flow
   \begin{equation}
      \label{flow_eta}
      \e{t\eta}(p,y) = (p-tDf(q), y)
   \end{equation}
   ($p\in T^*_qG$). Indeed, deriving \eqref{flow_eta} at $t=0$ in a standard chart $(p_i,q_i)$ of $T^*G$ one obtains in coordinates $\eta = (\d p_i,\d q_i, \d y)=(-\partial f/\partial q_i,0,0)$, whence
   \begin{equation}
      \omega(\eta,\cdot)=\d p_idq_i-\d q_idp_i+\tau(\d y,\cdot)= - df,
   \end{equation}
   i.e.~$\eta = \drag f$. The flow \eqref{flow_eta} is complete; since on the other hand $f$ is $H$-invariant, one knows \cite[Cor. 3]{Marsden:1974} that this flow passes to the quotient $X=\psi\inv(0)/H$ \eqref{Ind_H^GY}, where it is again the flow of $\drag f$ (computed on $X$). Therefore the latter is also complete. If further $f'$ is another function from $G/H$, one sees on \eqref{flow_eta} that is is constant along the flow, so that $\{f,f'\}=0$. Finally the equivariance of \eqref{pi_ind} shows that these functions constitute a $G$\nobreakdash-invariant space. So we have indeed a system of imprimitivity, whose base $B$ identifies with $G/H$ in the obvious manner.
\end{proof}

\subsection{The imprimitivity theorem}
Mackey's theorem \cites{Mackey:1949}[{}XI.14.19]{Fell:1988b} asserts that the presence of a transitive system of imprimitivity (in his sense) characterizes induced representations. Its symplectic analog will therefore consist in completing \eqref{f_ind} with a converse:

\begin{theo}
   \label{imprimitivity_theorem}
   Let $(X,\sigma,\Phi)$ be a hamiltonian $G$-space admitting a transitive system of imprimitivity $\LF$ with base $B=\pi(X)$\textup, and write $H$ for the stabilizer of some $b\in B$. Then there is a unique hamiltonian $H$-space $(Y,\tau,\Psi)$ such that
   \begin{equation}
      \label{X=Ind_H^GY}
      \textup{(a)}\quad X=\Ind_H^GY,
      \qquad\qquad\qquad
      \textup{(b)}\quad \pi\inv(b)=\pi\ind\inv(eH).
   \end{equation}
   Explicitly $Y=\pi\inv(b)/\F$\textup, i.e. $Y$ is the reduced space of $X$ at $b\in\LF^*$.
\end{theo}

\begin{remas} Condition (b) only serves to ensure the uniqueness of $Y$, which of course is understood up to isomorphism. Likewise by \eqref{X=Ind_H^GY} we mean ``there is a $J\in\Isom_G(X,\Ind_H^GY)$ that sends $\pi\inv(b)$ to $\pi\ind\inv(eH)$.''

The proof will detail the hamiltonian $H$-space structure of $\pi\inv(b)/\F$.
\end{remas}

\begin{proof}
   1. \emph{The level set $X_b:=\pi\inv(b)$ is a submanifold of $X$}. Indeed, the equivariance of $\pi:X\to B$ ensures that its derivative at $x\in X_b$ (\ref{transitive}ii) maps $\LG(x)$ onto $\LG(b)=T_bB$. So $\pi$ is a submersion, whence our claim.
   
   2. \emph{This submanifold is \emph{coisotropic}\textup, and more precisely\textup, the symplectic orthogonal of $T_xX_b$ is given by}
   \begin{equation}
      \label{orthogonal_T_xX_b}
      (T_xX_b)^\sigma = \LF(x)
   \end{equation}
   (which is isotropic since $\LF$ is commutative; note that \eqref{orthogonal_T_xX_b} is just what one would expect from (\ref{Ker_Im_DPhi(x)}a) if $\F$ were a Lie group and $\pi$ the moment map of its action on $X$). Indeed, the transpose of the exact sequence $0\to T_xX_b\to T_xX\to T_bB\to 0$ shows first that $(T_xX_b)^\sigma$ is the range of the injection
   \begin{equation}
      \label{j_x}
      \begin{tikzcd}
         j_x:T^*_bB\rar[hook] &T_xX
      \end{tikzcd}
   \end{equation}
   obtained by composing $T^*_bB\hookrightarrow T^*_xX$ with the isomorphism $T^*_xX\to T_xX$ given by the symplectic structure. On the other hand, each $f\in\LF$ is by construction the pull-back to $X$ of a function $\dot f$ on $B$, which is also $C^\infty$ since $\pi$ is a submersion. So the definition of $\drag f$ says that $(\drag f)(x) = j_x(-D\dot f(b))$, and proving \eqref{orthogonal_T_xX_b} boils down to showing that the map
   \begin{equation}
      \label{onto_T^*_bB}
      \LF\to T^*_bB,
      \qquad\qquad
      f\mapsto -D\dot f(b)
   \end{equation}
   is onto. To this end, observe that if $Z\in\LG$ and $g_t = \exp(tZ)$ then we have
   \begin{equation}
	   \begin{aligned}
	      Z\in\LH\quad 
	      &\iff\quad \<g_t(b),f\> = \<b,f\>
	      &&\forall f\in\LF,\ \forall t\\
	      &\iff\quad \dot f(g_t(b)) = \dot f(b)
	      &&\forall f\in\LF,\ \forall t\\
	      &\iff\quad \tfrac d{dt}\dot f(g_t(b)) = 0
	      &&\forall f\in\LF,\ \forall t\\
	      &\iff\quad D(\dot f\circ g_t)(b)(Z(b)) = 0\quad
	      &&\forall f\in\LF,\ \forall t\\
	      &\iff\quad D\dot f(b)(Z(b)) = 0
	      &&\forall f\in\LF
	   \end{aligned}
   \end{equation}
   since $\LF$ is $G$-invariant. Since $Z\in\LH$ is also equivalent to $Z(b)=0$, this shows that the $D\dot f(b)$ separate $T_bB$. Hence \eqref{onto_T^*_bB} is onto, and \eqref{orthogonal_T_xX_b} is proved.
   
   3. \emph{The orbit space $Y:=X_b/\F$ admits a unique manifold structure making $X_b\to Y$ a submersion}. (Note that \eqref{orthogonal_T_xX_b} implies that $\LF(x)\subset\LF(x)^\sigma = T_xX_b$, so that the action of $\F$ does indeed preserve $X_b$.) To see this, we note that what was said before \eqref{onto_T^*_bB} means that the action $\F\to\Diff(X_b)$ factors as
   \begin{equation}
      \begin{tikzcd}[row sep=0ex]
         \F\rar{\eqref{onto_T^*_bB}} & T^*_bB\rar &\Diff(X_b)\\
         f\rar[mapsto] & a\rar[mapsto] & \e{\hat a},
      \end{tikzcd}
      \label{factorization}
   \end{equation}
   where $\hat a$ denotes the vector field defined on $X_b$ by $\hat a(x) = j_x(a)$ \eqref{j_x}. Since \eqref{onto_T^*_bB} is onto, this shows that the $\F$-orbits in $X_b$ are in fact the orbits of an action of the (additive) \emph{Lie group} $T^*_bB$. Moreover, the definitions of $\Phi$ and $\hat a$ give, for all $Z\in\LG$,
   \begin{equation}
      \label{ahat_acheck}
      \<D\Phi(x)(\hat a(x)),Z\>
      = \sigma(\hat a(x), Z(x))
      = \<a,Z(b)\>
      = \<\check a, Z\>
   \end{equation}
   where $a\mapsto\check a$ is the transpose of $Z\mapsto Z(b)$, hence a bijection $T^*_bB\to\ann(\LH)$. Thus, $\Phi$ relates the field $\hat a$ on $X_b$ to the \emph{constant} vector field $\check a$. Therefore it intertwines the action $T^*_bB\to\Diff(X_b)$ with a mere action by translations:
   \begin{equation}
      \label{straightening}
      \Phi(\e{\hat a}(x)) = \Phi(x)+\check a.
   \end{equation}
   Since the latter action is free and proper, so is the former by \cite[§III.4.2, Prop. 5]{Bourbaki:1971a}; whence our assertion follows by \cite[§III.1.5, Prop. 10]{Bourbaki:1972}.
   
   4. \emph{$Y$ is naturally a hamiltonian $H$-space}. Indeed, $\sigma_{|X_b}$ vanishes precisely along the $\F$-orbits \eqref{orthogonal_T_xX_b}, hence is the pull-back of a symplectic form $\tau$ on $Y$ \cite[{}9.9]{Souriau:1970}. Likewise the action of $H$, which preserves $X_b$, passes to the quotient because it normalizes the image of \eqref{factorization}, since $\LF$ is $H$-invariant. Finally \eqref{straightening} shows that the resulting $H$-action on $Y$ admits a moment map $\Psi$, defined by the commutativity of the diagram
   \begin{equation}
      \begin{tikzcd}
         X_b \rar{\Phi}\dar[swap]{\F(\cdot)} &\LG^*\dar\\
         Y   \rar{\Psi}                &\LH^*\rlap{.}
      \end{tikzcd}
      \label{leaf_space}
   \end{equation}
   
   5. \emph{The induced manifold $\Ind_H^GY$ is isomorphic to $X$ as a hamiltonian $G$-space}. Indeed, let us adopt the notation of §\ref{KKS} and consider the maps $\varepsilon$ (resp.~$j$) from $G\x X_b$ to $X$ (resp.~to $T^*G\x Y$):
   \begin{equation}
      \label{epsilon_j}
      \varepsilon(q,x)=q(x),
      \qquad\text{resp.}\qquad
      j(q,x) = (q\Phi(x),\F(x)).
   \end{equation}
   It follows from \eqref{psi} and (\ref{factorization}--\ref{straightening}) that $j$ is a diffeomorphism $G\x X_b\to\psi\inv(0)$. Moreover one checks without trouble that $j$ maps fibers of $\varepsilon$ to $H$-orbits. Passing to the quotient, we obtain therefore a diffeomorphism $J$:
   \begin{equation}
      \begin{tikzcd}
         G\x X_b \rar{j}\dar[swap]{\varepsilon} 
         & \psi\inv(0) \rar[hook]\dar{\eqref{Ind_H^GY}}
         & T^*G\x Y\\
         X \rar[dashed]{J} & \Ind_H^GY
      \end{tikzcd}
      \label{J}
   \end{equation}
   which visibly is $G$-equivariant and maps $X_b$ onto $\pi\ind\inv(eH)$. To see that $J$ is symplectic, let us regard the variables which appear in \eqref{epsilon_j}:
   \begin{equation}
      \tilde x = q(x),\qquad
      y = \F(x),\qquad
      m = \Phi(x),\qquad
      n = (qm, y)
   \end{equation}
   as functions of $(q,x)\in G\x X_b$. Each tangent vector $(\d q,\d x)\in T_qG\x T_xX_b$ then gives (via the tangent map) a vector $(\d\tilde x, \d y, \d m, \d n)$ as well as an element $Z = q\inv\d q$ of $\LG$. That being said, the definition of $\omega = d\theta+\tau$, the definition of $\tau$ above, and formulas \cite[{}11.17$\sharp,\flat$]{Souriau:1970}, give
   \begin{equation}
	   \begin{aligned}
	      \llap{$\omega(\d n,\d'n)$}
	      &= \<\d m, Z'\> - \<\d'm,Z\> + \<m,[Z',Z]\> + \tau(\d y, \d'y)\\
	      &= \sigma(\d x,Z'(x)) 
	       - \sigma(\d'x,Z(x)) 
	       + \sigma(Z(x),Z'(x)) 
	       + \sigma\rlap{$(\d x,\d'x)$}\\
	      &= \sigma(\d x+Z(x),\d'x+Z'(x))\\
	      &= \sigma(\d\tilde x,\d'\tilde x).
	   \end{aligned}
   \end{equation}
   This shows that $j^*\omega = \varepsilon^*\sigma$, whence $J^*\sigma\ind = \sigma$ by \eqref{J} and by definition of $\sigma\ind$. Finally it is clear on \eqref{phi} and \eqref{epsilon_j} that $J^*\Phi\ind = \Phi$.
   
   6. \emph{There remains to establish the uniqueness assertion}. To this end, suppose that $(Y,\tau,\Psi)$ is any solution of the problem, so that one has a $G$-equivariant commutative diagram
   \begin{equation}
      \begin{tikzcd}
         X \rar\dar[swap]{\pi} & \Ind_H^GY \dar{\pi\ind}\\
         B \rar & G/H\rlap{,}
      \end{tikzcd}
      \label{uniqueness}
   \end{equation}
   where the horizontal arrows are isomorphisms and the bottom one maps $b$ to $eH$. This identifies $T^*_bB$ to $(\LG/\LH)^*\simeq\ann(\LH)$, and its action \eqref{factorization} on the fiber $X_b$ to the action of $\ann(\LH)$ by translation of $m$ in     
   \begin{equation}
      \label{fiber_pi_ind}
	   \begin{aligned}
	   	\pi\ind\inv(eH)
	   	&=\{H(m,y):(m,y)\in\LG^*\x Y, m\subh=\Psi(y)\}\\
	   	&\simeq \{(m,y)\in\LG^*\x Y: m\subh=\Psi(y)\}.
	   \end{aligned}
   \end{equation}
   (Here we use the observation that the elements of $\smash[b]{\pi\ind\inv(eH)}$, seen as $H$-orbits in $T^*G\x Y$ \eqref{Ind_H^GY}, each have a unique representative in $T^*_eG\x Y$.) Moreover the 2-form of $T^*G\x Y$ reduces on \eqref{fiber_pi_ind} to that of $Y$. Thus, $Y$ becomes identified with the quotient $X_b/\F$ we have constructed.
\end{proof}

\begin{exem}
   The theorem contains the equality $T^*(G/H)=\Ind_H^G\{0\}$, valid for any closed subgroup $H\subset G$. On the other hand, if we deprive $T^*(G/H)$ of its zero section, \eqref{imprimitivity_theorem} ceases to apply, because functions lifted from the base no longer have complete symplectic gradients (cf. \eqref{flow_eta}).
\end{exem}

\begin{rema}
   When both $G$ and $H$ are connected and simply connected, we mentioned in the Introduction that Theorem \eqref{imprimitivity_theorem} can be deduced from \cite{Xu:1991}. Briefly this is because the double fibration   \begin{equation}
      \begin{tikzcd}[row sep=tiny, column sep=small]
               & T^*G\dlar[swap]{\alpha\vphantom{\beta}}\drar{\beta}\\
         \LH^* &                               & \LG^*\rlap{$\x G/H$}
      \end{tikzcd}
      \label{bimodule}
   \end{equation}
   where $\alpha$ and $\beta$ map $p\in T^*_qG$ to $q\inv p\subh$, resp.~$(pq\inv,qH)$, is in Xu's terms an \emph{equivalence bimodule} between Poisson manifolds, whose \emph{complete symplectic realizations} $\Psi$ (resp.~$\Phi\x\pi$) arise from hamiltonian $H$-spaces \cite[Lem.~3.1]{Xu:1992}, resp.~from hamiltonian $G$-spaces with a system of imprimitivity \eqref{imprimitivity} based on $G/H$.
\end{rema}

\section{Morphisms}
Mackey completed his imprimitivity theorem with a bijection between the intertwining space of two $H$-modules, and \emph{part} of the intertwining space of the $G$-modules they induce \cites[Thm 6.4]{Mackey:1958}{Blattner:1962a}. The~symplectic analog is as follows. Given hamiltonian $G$-spaces $X_1$, $X_2$ with systems of imprimitivity over the same base $B$ (\ref{same_base}c), write $\Isom_B(X_1,X_2)$ for the set of all $J\in\Isom_G(X_1,X_2)$ such that the diagram
\begin{equation}
   \begin{tikzcd}[column sep=small,row sep=small]
      X_1\arrow{rr}{J}\drar[swap]{\pi_1} & & X_2\dlar{\pi_2}\\
      & B
   \end{tikzcd}
\end{equation}
commutes. Then we have:

\begin{theo}
   \label{intertwining}
   $\Isom_{G/H}(\Ind_H^GY_1,\Ind_H^GY_2)=\Isom_H(Y_1,Y_2)$.
\end{theo}

\begin{proof}
   Write $X_1$, $X_2$ for the manifolds induced by $Y_1$, $Y_2$ and $\pi_1$, $\pi_2$ for their projections \eqref{pi_ind} to $G/H$. Given $j\in\Isom_H(Y_1,Y_2)$, the symplectomorphism
   \begin{equation}
      \text{id}\x j: T^*G\x Y_1\to T^*G\x Y_2
   \end{equation}
   clearly passes to the quotients \eqref{Ind_H^GY} where it defines a $J\in\Isom_{G/H}(X_1,X_2)$. Conversely, any $J\in\Isom_{G/H}(X_1,X_2)$ maps $\smash{\pi_1\inv(eH)}$ onto $\smash{\pi_2\inv(eH)}$ while respecting the 2-forms and hence their characteristic foliations. There results an isomorphism between the leaf spaces, which Theorem \eqref{imprimitivity_theorem} shows are precisely $Y_1$ and $Y_2$. Finally one checks without trouble that the correspondences thus defined are each other's inverse.
\end{proof}

\begin{exems}
   \label{non_isomorphic}
   The theorem allows $\Isom_G(\Ind_H^GY_1,\Ind_H^GY_2)$ to be strictly larger than $\Isom_H(Y_1,Y_2)$. Thus:
   
   (a) \emph{Non-isomorphic spaces can induce isomorphic spaces}. Let $G=\CC\rtimes\mathbf{U}(1)$ be the displacement group of the plane, and $\Real$, $\Imaginary$ the linear forms `real part' and `imaginary part' regarded as coadjoint orbits of the subgroup $\CC$. Then $\Ind_\CC^G\Real$ and $\Ind_\CC^G\Imaginary$ are one and the same coadjoint orbit of $G$ (a cylinder).
   
   (b) \emph{A space without automorphisms can induce a space with automorphisms}. Let $G\tilde{\phantom\imath}=\CC\rtimes\RR$ be the universal covering of the previous $G$. Then $\smash{\Ind_\CC^{G\tilde{\phantom\imath}}\!\Real}$ is the universal covering of the previous cylinder, whose homotopy provides nontrivial automorphisms (deck transformations). Their presence reflects the fact that the representation $\smash{\Ind_\CC^{G\tilde{\phantom\imath}}\!\e{\mathrm i\Real}}$ is reducible \cite[p.\,189]{Bernat:1972} even though $G\tilde{\phantom\imath}$ is transitive on $\smash{\Ind_\CC^{G\tilde{\phantom\imath}}\!\Real}$.
\end{exems}

\section{The Homogeneous Case}
To apply the imprimitivity theorem \eqref{imprimitivity_theorem}, we need (transitive) systems of imprimitivity. When $X$ itself is homogeneous, these are all based on quotients $G/H$ by very special subgroups:

\begin{theo}
	\label{homogeneous_case}
	Let $(X,\sigma,\Phi)$ be a homogeneous hamiltonian $G$-space\textup, and $H$ a closed subgroup of $G$. Then $X$ admits a system of imprimitivity with base $G/H$ if and only if there is an $x\in X$ such that\textup, writing $\check c=\Phi(x)$\textup,
	\begin{enumerate}[\upshape(a)]
	   \item $H$ contains the stabilizer $G_x$\textup;
	   \item $\LH$ is \emph{coisotropic} at $\check c$\textup:\quad$\ann(\LH(\check c))\subset\LH$\textup;
	   \item $\LH$ satisfies the \emph{Pukánszky condition} at $\check c$\textup:\quad$\check c + \ann(\LH)\subset G(\check c)$.
	\end{enumerate}
\end{theo}

\begin{proof}
   Suppose $X$ admits a system \eqref{imprimitivity} with base $B=G/H$, and put $b=eH$. Then Theorem \eqref{imprimitivity_theorem} applies; so we have a diagram \eqref{uniqueness}, and \eqref{fiber_pi_ind} gives
   \begin{equation}
      \label{Phi(X_b)}
	   	\Phi(X_b)=\{m\in\LG^*: m\subh\in\Psi(Y)\}.
   \end{equation}
   Pick $x\in X_b$. Since $G$ is transitive on $X$, we have $H(x) = X_b$. So the orbit $H(\check c)$ equals \eqref{Phi(X_b)} and therefore contains $\check c + \ann(\LH)$, whence (c). Likewise its tangent space $\LH(\check c)$ contains $\ann(\LH)$, whence (b). Finally the equivariance of $\pi$ ensures (a).
   
   Conversely, suppose $x$ satisfies (a,b,c) and let us show that $X$ admits a system of imprimitivity with base $B=G/H$. By (a), $g(x)\mapsto gH$ well-defines an equivariant submersion $\pi:X\to B$ whose fiber at $b=eH$ is $X_b=H(x)$. Now put $\LF = \pi^*(C^\infty(B))$. Then we have the relations
   \begin{equation}
      \LH(x)^\sigma \subset \LH(x),
      \qquad
      \LH(x)^\sigma = \LF(x),
      \qquad
      \LF(x) \subset \LF(x)^\sigma.
   \end{equation}
   The first one comes from (b) and from $\sigma(\LH(x),Z(x))=\<\LH(\check c),Z\>$ \eqref{KKS_theorem}; the second one is proved like \eqref{orthogonal_T_xX_b}, and the third one follows. Moreover it is clear by transport of structure that the third relation still holds at all $x'\in X$, and the other two for all $x'$ in the $H$-orbit $X_b$. This shows that $\LF$ is a commutative subalgebra of $C^\infty(X)$, whose symplectic gradients $\eta$ are tangent to $X_b$. There remains to see that $\e{t\eta}(x)$ exists for all $t$. But \eqref{ahat_acheck} shows that $\Phi$ relates $\eta_{|X_b}$ to a constant field $\check a\in\ann(\LH)$, whose integral curve lies in $\Phi(X)$ by (c). Since $\Phi$ is a covering \eqref{KKS_theorem} this curve lifts to $X$, whence the conclusion.
\end{proof}

\begin{remas}
   \label{homogeneous_remarks}
   (a) Our proof exhibits $X$ as induced from the leaf space $Y$ of $H(x)$, which is a covering space of the coadjoint orbit $\smash{H(\check c\subh)}$ (\ref{elementary}c, \ref{orthogonal_T_xX_b}, \ref{leaf_space}). Note however that, as (\ref{non_isomorphic}a) shows, the relation $X=\smash{\Ind_H^GY}$ by itself does not characterize $Y$: this is the role of condition (\ref{X=Ind_H^GY}b).
   
   (b) The inclusion (\ref{homogeneous_case}b) is an equality iff $\LH$ is a \emph{real polarization}, iff $Y$ is zero-dimensional. More generally, the two terms of (\ref{homogeneous_case}b) can be the subalgebras $\LD$ and $\LE$ associated with a complex polarization \cite{Bernat:1972}. In that case, the observation that $G(\check c)$ is induced from $E$ goes back to \cite[~{}3.12ii]{Duval:1992}.

   (c) All of the above applies when $G$ is finite or discrete. Then (\ref{homogeneous_case}b,c) hold automatically, so a space $X=G/R$ is induced from any intermediate subgroup, $R\subset H\subset G$: \eqref{homogeneous_case} simply says that $G/R = \smash{\Ind_H^G(H/R)}$, and we recover the \emph{classical} notion of imprimitivity \cites[§3]{Neumann:2006}[{}2.3a]{Zimmer:1978}.
\end{remas}

\specialsection*{\textbf{Chapter II: The Normal Subgroup Analysis}}
\section{The Little Group Step}

A key observation in representation theory is that the restriction of an irreducible representation to a normal subgroup can only involve irreducibles which are conjugate under the ambient group \cites[Thm 1]{Clifford:1937a}[Lem.~9]{Blattner:1965}. The symplectic analog is clear: if $N$ is normal in $G$, then $G$ acts naturally in $\LN$ and $\LN^*$, and respects the partition of $\LN^*$ in $N$-orbits. So $G$ acts in the orbit space $\LN^*/N$, and the successive maps 
\begin{equation}
   \label{successive_projections}
   \begin{tikzcd}
      X
      \rar[swap]{\Phi\mathstrut}
      \arrow[loop below,in=-70,out=-115,looseness=4,swap,->]{}{G} &
      \smash{\LG\rlap{${}^*$}\ }
      \rar[swap]{(\cdot)\subn}
      \arrow[loop below,in=-70,out=-115,looseness=4,swap,->]{}{G} &
      \smash{\LN\rlap{${}^*$}\ }
      \rar[swap]{N(\cdot)}
      \arrow[loop below,in=-70,out=-115,looseness=4,swap,->]{}{G} &
      \smash{\LN^*/N}
      \arrow[loop below,in=-70,out=-115,looseness=4,swap,->]{}{G}
   \end{tikzcd}
\end{equation}
are $G$-equivariant. As a result we have the following triviality, in which the stabilizer $G_U$ (or sometimes $G_U/N$) is known as the \emph{little group}: 

\begin{theo}
   \label{little_group}
   Let $N\subset G$ be a closed normal subgroup. Then \eqref{successive_projections} maps any homogeneous hamiltonian $G$-space $(X,\sigma,\Phi)$ onto a $G$-orbit $\mathcal U=G(U)=G/G_U$ in $\LN^*/N$.\qed
\end{theo}

\section{The Induction Step}

In the setting of \eqref{little_group}, the representation-theoretic analogy next leads us to expect that $X$ should admit a system of imprimitivity based on $\mathcal U$, and hence should be induced \cites[Satz 133]{Speiser:1923}[Thm 2]{Blattner:1965}. This expectation is rewarded:

\begin{theo}
   \label{induction_step}
   Let $N\subset G$ be a closed normal subgroup\textup, and $U\in\LN^*/N$ an orbit such that $H:=G_U$ is closed. Then $H$ contains $N$\textup, and $Y\mapsto X=\smash{\Ind_H^GY}$ defines a bijection between \textup(isomorphism classes of\textup)
   \begin{enumerate}[\upshape(a)]
      \item homogeneous hamiltonian $G$-spaces $(X,\sigma,\Phi)$ such that $U\subset\Phi(X)\subn$\textup;
      \item homogeneous hamiltonian $H$-spaces $(\rlap{$Y$}\phantom{X},\tau,\Psi)$ such that $U=\Psi(Y)\subn$.
   \end{enumerate}
   Explicitly the inverse map sends $X$ to its reduced space at $U$ in \textup{\cite{Kazhdan:1978}'}s sense\textup, i.e.~$Y$ is the quotient of $\Phi(\cdot)_{\smash{|}\LN}\inv(U)$ by its characteristic foliation. Moreover $X$ is a coadjoint orbit of $G$ iff $Y$ is a coadjoint orbit of $H$.
\end{theo}

\begin{proof}
   The inclusion $N\subset H$ is clear, for $N$ acts trivially on $\LN^*/N$. To prove that $Y\mapsto X$ is onto, let $X$ be as in (a) and write $\pi:X\to\mathcal U=G(U)$ for the composition of the three maps \eqref{successive_projections}. We claim that $\pi^*(C^\infty(\mathcal U))$ is a system of imprimitivity on $X$. Indeed, this will result from Theorem \eqref{homogeneous_case} if we show that $H$ satisfies the following three relations, where $x\in X$ has successive images $\check c=\Phi(x)$, $c=\check c\subn$, and $U=N(c)$ under \eqref{successive_projections}:
   \begin{flalign}
      \stepcounter{equation}
      \label{H_contains_G_x}
      &H \text{ contains the stabilizer } G_x,&&\tag{\theequation a}\\[-.5ex]
      \label{H_coisotropic}
      &\LN_c(\check c)=\ann(\LH),&&\tag{\theequation b}\\[-.5ex]
      \label{Pukanszky}
      &N_c\o(\check c) = \check c + \ann(\LH)
      \qquad\qquad\text{(cf.~Pukánszky \cite[Lemma 2]{Pukanszky:1978}).}&& \tag{\theequation c}
   \end{flalign}
   The first relation is clear by equivariance of $\pi$. To prove the second, observe that by definition of the stabilizers $H=G_U$, $G_c$, and $N_c$ we have
   \begin{gather}
      \stepcounter{equation}
      H=NG_c,\tag{\theequation a} \label{H=NG_c}\\
      \LG_c = \ann(\LN(\check c)),\tag{\theequation b}\\
      \LN_c = \{Z\in\LN:Z(\check c)\in\ann(\LN)\},\tag{\theequation c}
   \end{gather}   
   whence
   $
   \ann(\LH)
   \overset{\text{a}}{=}\ann(\LN+\LG_c)
   =\ann(\LG_c)\cap\ann(\LN)
   \overset{\text{b}}{=}\LN(\check c)\cap\ann(\LN)
   \overset{\text{c}}{=}\nobreak\LN_c(\check c)
   $,
   as claimed. To see \eqref{Pukanszky}, note that \eqref{H_coisotropic} says that $\LN_c$ stabilizes $\smash{\check c\subh}$. Therefore so does $N_c\o$, which means that $N_c\o(\check c) \subset \check c + \ann(\LH)$. To prove the reverse inclusion, note that $\LN$ being an ideal implies $\ad(Z)^n(\LG)\subset [\LN_c,\LN]$ for all $Z\in\LN_c$ and all $n\geqslant2$. Since $c = \check c\subn$ vanishes on $[\LN_c,\LN]$ it follows that
   \begin{equation}
	   \begin{aligned}
	      \<\exp(Z)(\check c),Z'\>
	      &=\Bigl\<
	      \check c,\sum_{n=0}^\infty\frac{(-1)^n}{n!}\ad(Z)^n(Z')
	      \Bigr\>\\
	      &=\<\check c, Z'-[Z,Z']\>\\
	      &=\<\check c+Z(\check c), Z'\>
	   \end{aligned}
   \end{equation}
   for all $Z\in\LN_c$ and $Z'\in\LG$. Thus $\check c+\LN_c(\check c)$ is contained in $N_c\o(\check c)$, as claimed. So \eqref{homogeneous_case} applies, and $\pi^*(C^\infty(\mathcal U))$ is a system of imprimitivity. Let, then, $Y$ be the hamiltonian $H$-space provided by the imprimitivity theorem \eqref{imprimitivity_theorem}. It satisfies $X=\smash{\Ind_H^GY}$ and, by diagram \eqref{leaf_space}, $\smash{\Psi(Y)\subn}=\smash{\Phi(\pi\inv(U))\subh{}\subn}=U$, as claimed. Moreover we know that it is the quotient of $\pi\inv(U)$ by its characteristic foliation \eqref{orthogonal_T_xX_b}. Finally it is homogeneous (\ref{elementary}c), and is the coadjoint orbit $H(\check c\subh)$ when $X$ is the coadjoint orbit $G(\check c)$ (\ref{elementary}d, \ref{homogeneous_remarks}a).
   
   There remains to show that conversely, if $Y$ in (\ref{induction_step}b) is homogeneous or a coadjoint orbit, then so is $X=\smash{\Ind_H^GY}$. To this end, fix $y\in Y$, put $c=\Psi(y)\subn$, and pick a $\check c\in\LG^*$ such that $\check c\subh=\Psi(y)$. Now apply Lemma \eqref{Pukanszky} with the data $(G,X)$ replaced by $(G,G(\check c))$, resp.~$(H,Y)$. Taking into account that $\Psi$ is a covering, we obtain the relations
   \begin{equation}
      \label{Pukanszky_redux}
      N_c\o(\check c) = \check c + \ann(\LH),
      \qquad\text{resp.}\qquad
      N_c\o(y) = \{y\}.
   \end{equation}
   With that said, let $(qm_1,y_1)\in T^*G\x Y$ be a representative of an arbitrary element of the induced manifold \eqref{Ind_H^GY}. By transitivity, choose $h\in H$ such that $y_1=h(y)$, and put $m_2=h\inv(m_1)$. Then $(qhm_2,y)$ is another representative of the same element. Moreover the vanishing of \eqref{psi} implies that $m_2{}\subh=\Psi(y)=\check c\subh$, and then \eqref{Pukanszky_redux} shows that $m_2=n(\check c)$ for some $n\in N_c\o\subset H_y$. It follows that $(qhn\check c,y)$ is a third representative of the same element. So this element is the image by $g=qhn$ of the one represented by $(\check c, y)$, which shows that $G$ is transitive on $X$.
   
   Assume further that $Y$ is a coadjoint orbit, i.e.~the covering $\Psi$ is trivial, or in other words the inclusion $H_y\subset H_{\Psi(y)}$ is an equality. Writing $x=H(\check c,y)$ for the element of $\smash{\Ind_H^GY}$ we just mentioned (which $\Phi\ind$ visibly maps to $\check c$), we must show that the inclusion $G_x\subset G_{\check c}$ is also an equality. But $\check c$ was chosen so that
   \begin{equation}
      \check c\subh = \Psi(y)
      \qquad\text{and}\qquad
      \check c\subn = c.
   \end{equation}
   The second relation implies that $G_{\check c}\subset G_c\subset H$, and the first further implies that $G_{\check c}\subset H_{\Psi(y)} = H_y$. We conclude that $g\in G_{\check c}$ implies $g(x) = H(g\check c,y) = H(\check cg,y) = H(\check c, g(y)) = H(\check c, y) = x$, whence $G_{\check c}\subset G_x$.
\end{proof}

\begin{exem}
   The hypothesis in \eqref{induction_step} that $G_U$ is closed is not gratuitous, because $\LN^*/N$ needs not be Hausdorff. In fact it fails in well-known cases like the following: Let $G$ consist of all matrices of the form
   \begin{equation}
      g=
      \begin{pmatrix}
         a & b\\
         0 & c
      \end{pmatrix}
   \end{equation}
   where $a$ belongs to the diagonal torus $T\subset\mathbf U(2)$, $b$ is diagonal in $\Lie{gl}(2,\CC)$, and $c$ is in an irrational line $L\subset T$. If $N$ denotes the (`Mautner') subgroup $a=\bm1$ and $U$ the orbit of the value of the 1-form $\Re(\Tr(db))$ at the identity, then one finds that $G_U$ is the dense subgroup $\{g\in G:a\in L\}$.
\end{exem}

\section{The Obstruction Step}

In representation theory, the $H$-modules which upon restriction to a normal subgroup $N$, are a multiple of a given irreducible $N$-module, all arise by tensoring the latter with a projective $H/N$-module which ``kills the cocycle'' \cites[Thm 3]{Clifford:1937a}[{}XII.4.28]{Fell:1988b}. We develop here a symplectic analog of this idea.

\subsection{The flat bundle construction}
In the setting of (\ref{induction_step}b), let us fix $c\in U$. To improve on the (misled) expectation that $Y$ might be a product $U\x V$, let
\begin{equation}
   \label{U_tilde}
   \rho: \tilde U\to U
\end{equation}
be the covering $N/N\sub c\o\to N/N_c$ with principal group $\Gamma = N_c/N\sub c\o$. To any symplectic manifold $(V,\omega)$ with an $\omega$-preserving action of $\Gamma$, is then naturally associated the \emph{flat bundle}
\begin{equation}
   \label{flat_bundle}
   \tilde U\x_\Gamma V\to U
\end{equation}
with total space the set of orbits $[\tilde u,v]$ of the product action of $\Gamma$ on $\tilde U\x\nobreak V$. The $N$-action $n([\tilde u,v])=[n\tilde u, v]$ and the 2-form deduced from $\rho^*\sigma_{KKS}+\nobreak\omega$ (\ref{T*Q_KKS}b) by passage to the quotient make it a hamiltonian $N$-space with moment map \eqref{flat_bundle}, i.e., $[\tilde u,v]\mapsto\rho(\tilde u)$. 

\subsection{Barycentric decomposition and the Mackey obstruction}
Theorem \eqref{induction_step} used induction to reduce matters to the \emph{primary case}: a hamiltonian $H$-space $Y$ sitting above \emph{one} ($H$-stable) \emph{coadjoint orbit} of $N\subset H$. Using the flat bundle construction \eqref{flat_bundle}, we can further reduce to the case where we sit above \emph{one point}:

\begin{theo}
   \label{obstruction_step}
   Let $N\subset H$ be a closed normal subgroup\textup, $U=N(c)$ an $H$-stable coadjoint orbit of $N$\textup, and $\tilde U$ its covering \eqref{U_tilde} with group $\Gamma=N_c/N\sub c\o$. Then $V\mapsto Y=\smash{\tilde U}\x_\Gamma V$ defines a bijection between \textup(isomorphism classes of\textup)
   \begin{enumerate}[\upshape(a)]
      \item homogeneous hamiltonian $H$-spaces $(Y,\tau,\Psi)$ such that $\Psi(Y)\subn=U$\textup;
      \item homogeneous hamiltonian $H_c$-spaces $(V,\omega,\Upsilon)$ such that $\Upsilon(V)\subk=\{c\subk\}$.
   \end{enumerate}
   The inverse map sends $Y$ to the fiber $\Psi(\cdot)_{\smash{|}\LN}\inv(c)$\textup; this fiber is symplectic and $N\sub c\o$ acts trivially on it\textup, whence it carries an action of $\Gamma$. Moreover $Y$ is a coadjoint orbit iff $V$ is a coadjoint orbit. In addition if we pick $\breve c\in\LH^*$ such that $\breve c\subn = c$ then the formulas
   \begin{align}
      l(nN\sub c\o) &= lnl\inv N\sub c\o \tag{c}\\
      \varphi(nN\sub c\o) &= n(\breve c)_{|\LH_c} \tag{d}\\
      \theta(lN\sub c\o) &= (\breve c - l(\breve c))_{|\LH_c} 
      \rlap{$\quad\in\ann_{\LH_c^*}(\LN_c)\simeq(\LH_c/\LN_c)^*$}
      \tag{e}
   \end{align}
   define respectively a hamiltonian action of $H_c$ on $\tilde U$\textup, a moment map $\varphi$ for it\textup, and a non-equivariance cocycle such that $\varphi(l(\tilde u))=l(\varphi(\tilde u))+\theta(lN\sub c\o)$. The cohomology class $[\theta]\in H^1(H_c/N\sub c\o,(\LH_c/\LN_c)^*)$ is independent of $\breve c$ and vanishes if $c_{|\LN_c}=0$. Else its derivative is the cohomology class of the central extension
   \begin{equation}
      \begin{tikzcd}[column sep=scriptsize]
         0\rar &\LN_c/\Lie j\rar &\LH_c/\Lie j\rar &\LH_c/\LN_c\rar &0
      \end{tikzcd}
      \tag{f}
   \end{equation}
   where $\Lie j=\Ker(c_{|\LN_c})$\textup, and the categories \textup{(\ref{obstruction_step}a,b)} are equivalent to 
   \begin{equation}
      \{\text{homogeneous hamiltonian $(H_c/N\sub c\o,[-\theta])$-spaces}\},
      \tag{g}
   \end{equation}
   whereby we mean symplectic manifolds $(V,\omega)$ with a hamiltonian action of $H_c/N\sub c\o$ whose moment map $\psi$ can be chosen so that $\psi(\ell(v))=\ell(\psi(v))-\theta(\ell)$.
\end{theo}

\begin{proof}
   See \cite[Thm 2.2, Cor.~3.3, Cor.~3.9]{Iglesias-Zemmour:2015}.
\end{proof}

\begin{remas}
   (a) The class $[\theta]$ measures the obstruction to making $U$, on which $H$ acts symplectically, a hamiltonian $H$-space. We call it the (symplectic) \emph{Mackey obstruction} of $U$, and we call its derivative $[f]\in H^2(\LG_c/\LN_c,\RR)$, or the extension (\ref{obstruction_step}f) it encodes,  the \emph{infinitesimal Mackey obstruction} of $U$. Notice that it can be regarded as an extension of $\LH/\LN$, for $H=NH_c$ \eqref{H=NG_c} implies $H/N=H_c/N_c$. 
   
	(b) When $N$ is a Heisenberg group, Theorem \eqref{obstruction_step} is due to Souriau \cite[Thm 13.15]{Souriau:1970} who called it \emph{barycentric decomposition}: in applications, $U$ models the motions of a center of mass and $V$ the proper motions around it. In this case we have $\Gamma = 0$ and so $Y = U\x V$. For examples where $\Gamma\ne0$ and $Y$ doesn't split as a product we refer to \cite[{}4.5, 4.18]{Iglesias-Zemmour:2015}.
   
   (c) The simplest case with nonzero Mackey obstruction is when $H$ is a Heisenberg group and $N$ is its center. In contrast, the following symplectic analog of \cite[Prop.~XII.5.5]{Fell:1988b} gives conditions where it must be zero:
\end{remas}

\begin{prop}
   \label{semidirect}
   In the setting of \eqref{obstruction_step}\textup, assume that $U$ is a point-orbit~$\{c\}$ and that $N$ is a connected semidirect factor in $H$. Then the Mackey obstruction of $U$ vanishes.
\end{prop}

\begin{proof}
   First assume only that $N$ is connected and $H(c)=\{c\}$. Then we have $(H_c,N_c,N\sub c\o,\tilde U) = (H,N,N,U)$, and $\breve c\subn = c$ implies $\<\LN(\breve c),\LH\>=\<\breve c,[\LH,\LN]\>=\<c,[\LN,\LH]\> = \<\LH(c),\LN\>=0$. Hence $\LN(\breve c)=0$ and therefore
   \begin{equation}
      \label{N_fixed}
      N(\breve c)=\{\breve c\}.
   \end{equation}
   (This indeed is what makes (\ref{obstruction_step}d): $\varphi(c)=\breve c$ well-defined.) Now add the semidirect factor assumption: $H$ has another closed subgroup $S$ with $H=NS$ and $N\cap S =\{e\}$. Then $\LH=\LN\oplus\LS$, and we can choose the extension $\breve c$ so that $\breve c\subs = 0$. Then $s\in S$ implies $\<s(\breve c),\nu+\sigma\>=\<s(c),\nu\>+\<\breve c,s\inv\sigma s\>=\<c,\nu\>=\<\breve c,\nu+\sigma\>$ for all $\nu+\sigma\in\LH$, whence
   \begin{equation}
      \label{S_fixed}
      S(\breve c)=\{\breve c\}.
   \end{equation}
   Together \eqref{N_fixed} and \eqref{S_fixed} show that $\breve c$ is $H$-invariant, i.e., the cocycle (\ref{obstruction_step}e) is identically zero.
\end{proof}

\section{Synopsis}

Putting Theorems (\ref{little_group}, \ref{induction_step}, \ref{obstruction_step}) together and observing that $c\in\LN^*$ has the same stabilizer in $G$ as in $H=G_U$, we obtain the following summary result:

\begin{theo}
   \label{synopsis}
   Let $N\subset G$ be a closed normal subgroup and $(X,\sigma,\Phi)$ a~homogeneous hamiltonian $G$-space. Then $\smash{\Phi(X)\subn}$ sits above a $G$-orbit $G(U)$ in~$\LN^*/N$. If the stabilizer $G_U$ is closed and $U=N(c)$ has Mackey obstruction $[\theta]$\textup, then there is a unique homogeneous hamiltonian $(G_c/N\sub c\o,[-\theta])$-space $V$ such~that
   \begin{equation}
      X = \Ind_{G_U}^G (\tilde U\x_\Gamma V)
   \end{equation}
   where $\tilde U$ is the covering \eqref{U_tilde} with group $\Gamma = N_c/N\sub c\o$. Every homogeneous hamiltonian $(G_c/N\sub c\o,[-\theta])$-space arises in this way\textup, and $X$ is a coadjoint orbit iff $V$ is an orbit of the affine coadjoint action defined by $-\theta$ \textup{\cite[{}11.28]{Souriau:1970}}.\qed
\end{theo}

\specialsection*{\textbf{Chapter III: Applications}}
\section{Symplectic Mackey-Wigner Theory}

When the normal subgroup of \eqref{induction_step} is connected abelian, its coadjoint orbits are points and the obstruction step \eqref{obstruction_step} becomes tautological. Taking \eqref{semidirect} into account, we see that the theory then boils down (in the case of coadjoint orbits) to the following:

\begin{theo}
   \label{Mackey_Wigner}
   Let $A\subset G$ be a closed connected abelian normal subgroup and $X=G(\check c)$ a coadjoint orbit of $G$. Then there is a unique coadjoint orbit $Y$ of the stabilizer $H=G_c$ of $c = \check c\suba$\textup, namely $Y=H(\check c\subh)$\textup, such that
   \begin{equation}
      \label{Mackey_Wigner_formula}
      X=\Ind_H^GY
      \qquad\quad\text{and}\qquad\quad
      Y\suba = \{c\}.
   \end{equation}
   Moreover $Y$ is also the reduced space $\pi\inv(c)/A$ where $\pi$ is the projection $X\to\nobreak\LA^*$. In particular $A$ acts trivially on $Y$\textup, and if $A$ is a semidirect factor in $G$ then $Y$ is \textup(lifted from\textup) a coadjoint orbit of $H/A$.\qed
\end{theo}

\begin{exem}[Poincaré orbits]
   To take the case that gave the theory its name, consider the Poincaré group
   \begin{equation}
      G=\left\{g=
      \begin{pmatrix}
         L & C\\
         0 & 1
      \end{pmatrix}:
      \begin{array}{l}
         \rlap{$L$}\phantom C\in\mathbf{SO}(3,1)\o\\
         C\in\RR^{3,1}
      \end{array}
      \right\},
   \end{equation}
   a semidirect product of $A = \RR^{3,1}$ $(L=\bm1)$ with the Lorentz group $(C=0)$. Then $\LA^*$ identifies with $\RR^{3,1}$ where $G$ acts by $g(P) = LP$, and \eqref{Mackey_Wigner} classifies the coadjoint orbits $X$ of $G$ in terms of the possible orbits $X\suba$ and $Y$, thus:  
	\begin{enumerate}[\upshape(a)]
	   \item $X\suba$ is half a timelike hyperboloid and $Y$ is a coadjoint orbit of $\mathbf{SO}(3)$,
	   \item $X\suba$ is a half-cone and $Y$ is a coadjoint orbit of the euclidean group $\mathbf E(2)$,
	   \item $X\suba$ is a spacelike hyperboloid and $Y$ is a coadjoint orbit of $\mathbf{SL}(2,\RR)$,
	   \item $X\suba$ is the origin and $Y(=X)$ is a coadjoint orbit of the Lorentz group.
	\end{enumerate}
	This classification \cites{Souriau:1966a}[§14]{Souriau:1970} is of course completely parallel with the representation theory of $G$ as worked out by Wigner \cite[{}XII.8.17]{Fell:1988b}.
\end{exem}

\begin{remas}
   \label{Mackey_Wigner_remarks}
   (a) Theorem \eqref{Mackey_Wigner} can also be proved by applying the imprimitivity theorem \eqref{imprimitivity_theorem} directly to the (finite-dimensional!) system of imprimitivity $\LF=\{\<\cdot,Z\>:Z\in\LA\}$. Conversely, \eqref{imprimitivity_theorem} can \emph{formally} be regarded as an application of \eqref{Mackey_Wigner} to the (usually infinite-dimensional) group  $\F\rtimes G$ with normal abelian subgroup $\F$. This is in effect how \eqref{imprimitivity_theorem} was discovered.
   
   (b) The theorem holds as stated and proved not only when $A$ is abelian, but slightly more generally when $A$ is \emph{$X$-abelian} in the sense that $[\LA,\LA]$ lies in the ``extraneous'' ideal $\ann(X)=\Ker(Z\mapsto\<\cdot,Z\>)$.
   
   (c) In the context of \eqref{Mackey_Wigner}, it is useful to introduce the notation $\LE\orth = \ann(\LE(\check c))$ (cf.~(\ref{homogeneous_case}b)), in terms of which we have $\LA\subset\LA\orth$, $\LH=\LA\orth$, $\LH\orth\subset\LH$~and
   \begin{equation}
      \label{dimensions}
      \dim(G/H) = \dim(\LA(\check c)),\qquad\quad
      \dim(Y)=\dim(\LH/\LH\orth).
   \end{equation}
   (Indeed, one checks immediately that if $\LA$ is an ideal of a subalgebra $\LB$, then $\LA\orth\cap\LB$ is the stabilizer of $\check c\suba$ in $\LB$. Now apply this to the pairs $\LA, \LG$ and $\LH,\LH$.) Thus we see that $Y$ will be zero-dimensional just when $\LH$ is a \emph{polarization} ($\LH=\LH\orth$). At the other extreme, \eqref{Mackey_Wigner_formula} will boil down to $X=\Ind_G^GX$ just when $A$ acts trivially on $X$, i.e., when the $X$-abelian ideal $\LA$ is \emph{$X$-central} in the sense that $[\LG,\LA]\subset\ann(X)$.
\end{remas}

\section{Symplectic Kirillov-Bernat Theory}

In this section we assume that $G$ is an exponential Lie group. This means that $\exp:\LG\to G$ is a diffeomorphism, or equivalently \cite{Bernat:1972} that
\begin{enumerate}[(a)]
   \item $G$ is connected, simply connected, and solvable; and
   \item $\ad(Z)$ has no purely imaginary eigenvalues for $Z\in\LG$.
\end{enumerate}
Takenouchi \cite{Takenouchi:1957} proved that every irreducible unitary $G$-module is monomial, i.e.~induced from a character. The analog is that every coadjoint orbit is induced from a point-orbit:

\begin{theo}
   \label{Kirillov_Bernat}
   Let $X=G(\check c)$ be a coadjoint orbit of the exponential Lie group $G$. Then $G$ admits a closed connected subgroup $H$ such that
   \begin{equation}
      \label{monomial}
      X = \Ind_H^G\{\check c\subh\}.
   \end{equation}
   Its Lie algebra $\LH$ is a polarization satisfying Pukánszky's condition \textup{(\ref{homogeneous_case}d, \ref{homogeneous_remarks}b)}.
\end{theo}

\begin{proof}
   Takenouchi's key lemma \cite[§3]{Takenouchi:1957} ensures that $\LG/\ann(X)$ admits an abelian ideal which is not central. Taking its preimage in $\LG$, we obtain an $X$-abelian ideal $\LA$ of $\LG$ which is not $X$-central (\ref{Mackey_Wigner_remarks}b,c). Now Theorem \eqref{Mackey_Wigner} implies $X=\smash{\Ind_{G_1}^GX_1}$ where $G_1$ is the stabilizer of $\check c\suba$ and $X_1=G_1(\check c_{\smash{|}\LG_1})$. Moreover $G_1$ has smaller dimension than $G$ \eqref{dimensions} and is again exponential \cite[p.\,4]{Bernat:1972}. So we may iterate the process to obtain decreasing $G_i$ such~that
   \begin{equation}
      X = \Ind_{G_1}^G\cdots\Ind_{G_i}^{G_{i-1}}X_i=\Ind_{G_i}^GX_i
   \end{equation}
   where the dimension of $X_i=G_i(\check c_{\smash{|}\LG_i})$ decreases at each step (\ref{elementary}a, \ref{elementary}e). Ultimately we arrive at a point-orbit of $H=G_n$ say, such that \eqref{monomial} holds. Conversely if \eqref{monomial} holds, we have $\LH\subset\LH\orth$ since $\{\check c\subh\}$ is a point-orbit, and this inclusion is an equality by dimension (\ref{elementary}a). Moreover (\ref{elementary}b) says that $X$ is the only coadjoint orbit of $G$ such that $X\subh$ contains $\{\check c\subh\}$, and this means that $\check c+\ann(\LH)\subset X$.
\end{proof}

This proof can be regarded as a method to construct Pukánszky polarizations. Put algebraically it runs as follows:

\begin{algo}
   To obtain a Pukánszky polarization at $\check c$\textup, start with $\LG_0=\LG$. If $\LG_i\orth=\LG_i$\textup, we are done. Else\textup, pick an ideal $\LA_i$ of $\LG_i$ such that $\LA_i\subset\LA_i\orth\not\supset\LG_i$\textup, put $\LG_{i+1}=\LG_i\cap\LA_i\orth$\textup, and repeat.
\end{algo}

\begin{remas}
	(a) This method gives more polarizations than M.~Vergne's \cite[p.\,87]{Bernat:1972}. For example it gives the $\LH$ of \cite[p.\,88]{Bernat:1972}: in the notation there, choose $\LA_0=\operatorname{span}(x_3,x_5)$ and then $\LA_1=\operatorname{span}(x_1,x_3,x_5)(=\LH)$.
	
	(b) This method still doesn't give all Pukánszky polarizations. For example it misses the $\LB_1$ of \cite[p.\,313]{Benoist:1990a}: indeed, the latter contains no noncentral ideal (``$\LB_1\cap C_2(\LG)\subset C_1(\LG)$''), whereas our $\LH$ always contains $\LA_0$. In fact, taking orthogonals in the relation $\LH\subset\LG_{i+1}\subset\LA_i\orth$ shows that we always have $\LA_i\subset\LH\subset\LA_i\orth$ and $\LG_i\orth\subset\LH\subset\LG_i$ for all $i$.
\end{remas}

\section{Symplectic Parabolic Induction}

Mackey's normal subgroup analysis is no help in studying simple groups, which lack nontrivial normal subgroups, nor more generally semisimple or reductive groups. Yet, as \cite[§§15-16]{Mackey:1952} observed at once, induction does play a key role in their representation theory.

This can be translated geometrically. Following Vogan \cite{Vogan:2000a}, let us call $G$ \emph{reductive} if there is a homomorphism $\eta:G\to\mathbf{GL}(n,\RR)$ with finite kernel and $\Theta$-stable image, where $\Theta(g)={}^tg\inv$. Then $\LG$ identifies with a Lie algebra of matrices, $\LG^*$ identifies with $\LG$ by means of the trace form $\<Z,Z'\> = \Tr(ZZ')$, and every $x\in\LG^*$ has a unique Jordan decomposition
\begin{equation}
   x = x_{\,\textup h} + x_{\,\textup e} + x_{\,\textup n},
   \qquad\qquad
   x_{\,\textup h}, x_{\,\textup e}, x_{\,\textup n}\in\LG^*,
\end{equation}
where the matrix $x_{\,\textup h}$ is hyperbolic (diagonalizable with real eigenvalues), $x_{\,\textup e}$ elliptic (diagonalizable with imaginary eigenvalues) and $x_{\,\textup n}$ nilpotent, and $x_{\,\textup h}$, $x_{\,\textup e}$ and $x_{\,\textup n}$ commute. (For all this, see \cite[§2]{Vogan:2000a}.)

With that said, we have:

\begin{theo}
   \label{parabolic_induction}
   Let $X=G(x)$ be a coadjoint orbit of the reductive Lie group $G$. Let $\LU$ be the sum of eigenspaces belonging to positive eigenvalues of $\ad(x_{\,\textup h})$ and $U=\exp(\LU)$. Then $Q=G_{x_{\,\textup h}}U$ is a closed subgroup of $G$\textup, and one has
   \begin{equation}
      \label{Ind_Q^GY}
      X=\Ind_Q^GY
      \qquad\text{where}\qquad
      Y = Q(x_{\smash{|}\LQ}).
   \end{equation}
   Moreover $Y\subu=\{0\}$\textup, so $Y$ comes from an orbit of the quotient $G_{x_{\,\textup h}}=Q/U$.
\end{theo}

\begin{proof}
   The fact that $Q$ is closed and the semidirect product of $G_{x_{\,\textup h}}$ and $U$ is in \cite[Prop.~2.15]{Vogan:2000a}. So \eqref{Ind_Q^GY} will result from \eqref{homogeneous_case} and the imprimitivity theorem \eqref{imprimitivity_theorem}, if we show that $Q$ satisfies the following three relations:
   \begin{flalign}
      \stepcounter{equation}
      \label{Q_contains_G_x}
      &Q \text{ contains the stabilizer } G_x,&&\tag{\theequation a}\\[-.5ex]
      \label{Q_coisotropic}
      &\LU(x)=\ann(\LQ),&&\tag{\theequation b}\\[-.5ex]
      \label{Harish_Chandra}
      &U(x) = x + \ann(\LQ)
      \qquad\text{(cf.~Harish-Chandra \cite[Lemma~8]{Harish-Chandra:1954}).}&& \tag{\theequation c}
   \end{flalign}
   The first relation is clear: if $g\in G_x$, then $\eta(g)$ commutes with $x$ and~hence with $x_{\,\textup h}$ by a well-known property of the Jordan decomposition; so $G_x\!\subset G_{x_{\,\textup h}}$. Next, write $\LG^a\subset\LG$ for the eigenspace belonging to eigenvalue $a$ of $\ad(x_{\,\textup h})$. Then we have $\LG=\bigoplus_{a\in\RR}\LG^a$ and  $\LU=\bigoplus_{a>0}\LG^a$. The Jacobi identity gives
   \begin{equation}
      \label{Jacobi}
      [\LG^a,\LG^b]\subset\LG^{a+b},
   \end{equation}
   and invariance of the trace form implies $0=\<[x_{\,\textup h},Z],Z'\>+\<Z,[x_{\,\textup h},Z']\>=(a+b)\<Z,Z'\>$ for all $(Z,Z')\in\LG^a\x\LG^b$. Hence
   \begin{equation}
      \label{invariance_trace}
      \<\cdot,\cdot\>_{|\LG^a\x\LG^b}
      \quad\text{is}\quad
      \begin{cases}
         \text{zero} & \text{if } a+b\ne0,\\
         \text{non-degenerate} & \text{if } a+b=0.
      \end{cases}
   \end{equation}
   Now \eqref{Jacobi} shows that $\LQ=\LG^0\oplus\LU$, and \eqref{invariance_trace} that $\ann(\LQ)=\LU$. So \eqref{Q_coisotropic} comes down to seeing that $\LU(x)=\LU$, or in other words, that $\ad(x)$ maps $\LU$ onto $\LU$. And indeed it maps $\LU$ \emph{into} $\LU$ by \eqref{Jacobi} since $x\in\LG^0$, and \emph{onto} $\LU$ since $\Ker(\ad(x))=\LG_x\subset\LG_{x_{\,\textup h}}=\LG^0$. From \eqref{Q_coisotropic} we deduce first that $\LU$ and hence $U$ stabilize the projection $x_{\smash{|}\LQ}$, so that $U(x)\subset x+\ann(\LQ)$, and second that $U(x)$ is open in $x+\ann(\LQ)$. But $U$ is nilpotent \eqref{Jacobi}; so $U(x)$ is also closed \cite[p.\,7]{Bernat:1972}, and \eqref{Harish_Chandra} is proved.
   
   As to the theorem's last assertion, it suffices to note that $\smash{Y\subu}$ is the $Q$-orbit of $x\subu$ in $\LU^*$, and that $x\subu=0$ by \eqref{invariance_trace} since $x\in\LG^0$.
\end{proof}

\begin{rema}
   $Q$ is a parabolic subgroup of $G$ with Levi factor $G_{x_{\,\textup h}}$. One should not confuse induction in the sense \eqref{Ind_Q^GY} with Lusztig-Spaltenstein induction \cites{Lusztig:1979}[Chap.~7]{Collingwood:1993}, which uses a smaller parabolic and is tailored for nilpotent orbits.
\end{rema}

%%%%%%%%%%%%%%%%%%%%%%%%%%%%%%%%%%%%%%%%%%%%%%%%%%%%%%%%%%%%%%%%%%%%%%%%
% \setlength{\labelalphawidth}{2.9em}
\setlength{\labelalphawidth}{3.2em}

%%% restore this:
\let\l\polishl

% \nocite{Auslander:1971a,Vergne:2002}

\printbibliography
%%%%%%%%%%%%%%%%%%%%%%%%%%%%%%%%%%%%%%%%%%%%%%%%%%%%%%%%%%%%%%%%%%%%%%%%

\end{document}